\documentclass[12pt,a4paper]{amsart}
\usepackage{amsfonts}
\usepackage{amssymb}
\usepackage{amsthm}
\usepackage[centertags]{amsmath}
\usepackage{newlfont}
\usepackage{graphicx}

\newtheorem{theorem}{Theorem}
\newtheorem{proposition}[theorem]{Proposition}
\newtheorem{corollary}[theorem]{Corollary}
\newtheorem{definition}[theorem]{Definition}

\def\r{\mathbb R}
\def\s{\mathbb S}

\date{}

\begin{document}
\title{Capillary surfaces in a cone }
\author{Rafael L\'opez}
 \address{Departamento de Geometr\'{\i}a y Topolog\'{\i}a\\
Universidad de Granada\\
18071 Granada, Spain\\}
 \email{ rcamino@ugr.es}
 \author{Juncheol Pyo}
\address{Department of Mathematics\\ Pusan National University\\ Busan 609-735, Korea}
\email{jcpyo@pusan.ac.kr}

\begin{abstract}We show that a   capillary surface in  a solid cone, that is, a surface that has constant mean curvature and the boundary of surface meets the boundary of the cone with a constant angle, is radially graphical if the mean curvature is non-positive with respect to the Gauss map pointing toward the domain bounded by the surface and the boundary of the cone. In the particular case that  the cone is circular, we prove that the surface is a spherical cap or a planar disc. The proofs are based on an extension of the Alexandrov reflection method by using inversions about spheres.
\end{abstract}


 \maketitle

\section{Introduction and results}

Consider a given amount of liquid deposited on a solid substrate with conical shape: see Figure \ref{fig1}. In absence of gravity, we analyse the equilibrium configurations when the drop reaches a state of critical interfacial area. {To be precise, let $(x,y,z)$ be the usual coordinates of $\r^3$, where $z$ indicates the vertical direction}. Let $\s^2$ be the unit sphere  centered at the origin $O$ of $\r^3$ and let $D\subset\s^2$ be a simply-connected domain of a hemisphere of $\s^2$. Denote by $C_D$ the cone defined by the union of all half-lines starting at $O$ passing by all points of $D$ and $C_\Gamma$ the boundary of $C_D$, that is, the union of the rays that start at $O$ through all points of $\Gamma=\partial D$. The point $O$ is called the vertex of the cone. The cone $C_D$ is called circular of opening angle $2\varphi\in (0,\pi)$ if $\Gamma\subset\s^2$ is a circle of radius $\sin(\varphi)$.  A liquid drop in the cone is viewed as  the closure of a domain $\Omega\subset\r^3$ which is confined in the cone $C_D$ and whose boundary $\partial \Omega$ intersects $C_\Gamma$. The  boundary $\partial\Omega$ of $\Omega$ is written by $\partial\Omega=T\cup S$, where $T\subset C_\Gamma$, $S=\partial\Omega\setminus T$ and $\partial\Omega$ is not smooth along its boundary $\partial S$. Physically,  $T$ is the wetted region  by the liquid drop $\overline{\Omega}$ in $C_\Gamma$.

 \begin{figure}[hbtp]
\begin{center}
\includegraphics[ width=.25\textwidth]{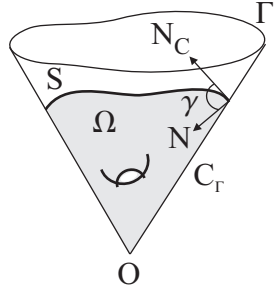}
\end{center}
\caption{A drop $\Omega$ supported on a cone $C_\Gamma$. The interface $S$ is a capillary surface, which means that $S$ has constant mean curvature and the angle $\gamma$ between the unit normal vectors $N$ and $N_C$ of $S$ and $C_\Gamma$ respectively, is constant along the boundary curve $\partial S$}\label{fig1}
\end{figure}

  In absence of gravity, and by the Young-Laplace equation, the shape of a liquid drop in equilibrium is characterized by the next two properties: 1) the surface $S$ has constant mean curvature  and 2) the angle $\gamma$ which $S$ meets $C_\Gamma$ is constant. Here  $\cos\gamma=\langle N,N_C\rangle$ along $\partial S$, where $N$ and $N_C$ denote the unit normal vector fields of $S$ and $C_\Gamma$ that point towards $\Omega$ and $C_D$ respectively.  With the above notation, we give the next definition:

\begin{definition} A capillary surface in a cone $C_D$ is a compact embedded surface of constant mean curvature such that $int(S)\subset C_D$, $\partial S\subset C_\Gamma$ and which meets the cone $C_\Gamma$ in a constant angle.
\end{definition}

In principle, a capillary surface in a cone can have higher topology, as well as, any number of boundary components. Physically capillary surfaces arise as the interface of an incompressible liquid in a container (see \cite{fi} and references therein). In the context of liquid drops in a cone we refer to \cite{kn}.

Examples of capillary surfaces in a circular cone $C_D$ are some pieces of spheres and planar discs (see Figure \ref{fig2}). Indeed, take a round disc $D\subset\s^2$ centered at the north pole and radius $\sin(\varphi)\in (0,1)$. Consider a sphere $\Sigma$ centered at the positive  $z$-axis of radius sufficiently large so   that  $\Sigma$ intersects $C_\Gamma$. Then $S=\Sigma\cap C_D$ is a capillary surface in $C_D$.  Depending on the values of $\gamma$ and $\varphi$, we have: if $\gamma>\pi/2+\varphi$, $S$ is a concave interface (case (a));  if $\pi/2-\varphi\leq \gamma<\pi/2+\varphi$, $S$  is a convex interface (case (b));  if $\gamma<\pi/2-\varphi$, $S$ is forbed by two spherical caps (case (c)); and finally, if $\gamma=\pi/2+\varphi$, we have a flat interface (case (d)).

\begin{figure}[hbtp]
\begin{center}
\includegraphics[width=.8\textwidth]{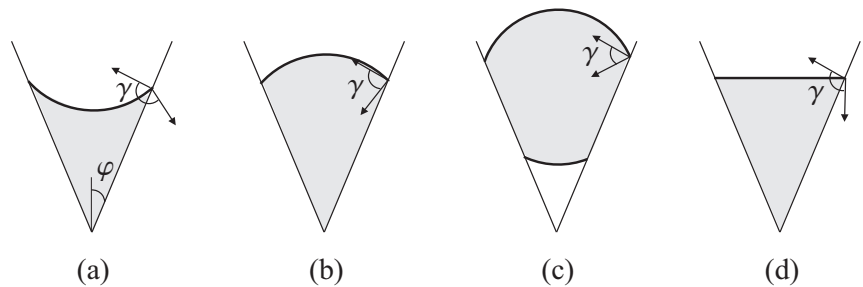}
\end{center}
\caption{Different configurations of capillary spheres and discs on a circular cone}\label{fig2}
\end{figure}

For a general cone, a first set of examples of capillary surfaces is considering non-parametric surfaces, that is, surfaces with a one-to-one central projection on $D$. We say then that the surface $S$ is a radial graph, which implies that any half-line starting at $O$ intersects $S$   one point at most. Results of existence and uniqueness of radial graphs with constant mean curvature have been obtained in \cite{lo,ra,sc,se,ta}. Radial graphs and capillary surfaces in a cone are the analogous examples  of  (vertical) graphs on a plane  and capillary surfaces on a vertical cylinder if we move the vertex $O$ of the cone at infinity. When the cone $C_\Gamma$ is convex, Choe and Park have shown that if  a parametric capillary surface $S$ meets orthogonally
$C_\Gamma$, then $S$ is part of a sphere \cite{cp}. Ros and Vergasta proved that $S$ has some topological restrictions if $S$ is stable \cite{rv}. We also point out  the results obtained by Vogel in \cite{vo}. He considered a drop $\Omega$ in  a cone $C_D$ such that its adherence $\overline{\Omega}$ contains the vertex $O$. Then he proved that the surface $S$ is a radial graph if the mean curvature $H$ is non-positive with respect to the  unit normal vector field pointing to $\Omega$.

This is our starting point. We ask  about  the shape of a capillary surface in a cone and how the boundary of the surface imposes restrictions to the shape of the surface. For example, we pose the next questions:
\begin{enumerate}
\item Does exist a capillary surface in a cone $C_D$ whose boundary is nullhomologous in $C_\Gamma\setminus\{O\}$? We do not know explicit examples and we hope that the answer is no at least if the surface is stable.
\item If the boundary $\partial S$ is homologous to $\Gamma$ in $C_\Gamma\setminus\{O\}$, is $S$ a radial graph?
\item Under what conditions  imply that  a capillary surface is part of a sphere or a plane?
\item   Under what conditions  imply that  a capillary surface  is a bridge, that is, the surface has the topology of a portion of a cylinder  bounded by two simple closed curves?
\end{enumerate}

We extend the  results of \cite{vo}  assuming only the non-positivity of the mean curvature $H$. Recall that a compact embedded surface $S$ contained in the cone $C_D$ and $\partial S\subset C_\Gamma$ defines a closed   $3$-dimensional  domain $\Omega\subset C_D$ such that its boundary $\partial \Omega$ is $\partial\Omega=T\cup S$ with $T\subset C_\Gamma$.

With the above notation, our result is the following:

\begin{theorem}\label{tmain}  Let  $S$ be a  connected  capillary surface supported in  a cone $C_D$. Let  fix the unit normal vector field $N$ of $S$ pointing towards the bounded domain $\Omega$. If $H\leq 0$, then   $S$ is a radial graph and the boundary $\partial S$ has only one connected component which is homologous to $\Gamma$ in $C_{\Gamma}\setminus \{O\}$. In the particular case that the cone is circular, $S$ is a planar disc or a spherical cap.
\end{theorem}

 In particular, Theorem \ref{tmain} informs us that the non-positivity of the mean curvature of a capillary surface $S$ in a cone implies that $S$ is a topological disc.

 \section{The spherical reflection method}

Alexandrov introduced the reflection method for surfaces with constant mean curvature  proving that a sphere is the only closed embedded surface of constant mean curvature \cite{al}. The idea is using the very surface as a barrier in a process of reflection about a one parameter family of parallel planes and  applying  the maximum principle of the mean curvature equation. The Alexandrov method by planes has two steps. Firstly it is proved that for each spatial direction $v\in\r^3$, there exists an orthogonal plane $\Pi_v$  to $v$ such that the surface is invariant by reflections about $\Pi_v$. The second step  proves that the only compact surface with this property is a sphere.

In this paper we use the \emph{spherical reflection method}  introduced by McCuan \cite{mc2}. Studying capillary surfaces in a wedge, McCuan modifies successfully the   Alexandrov reflection method   by inversions about a one parameter  family of spheres with the same center. Although  an inversion does not preserve the mean curvature, one can have a certain control of the mean curvature of the inverted surface in order to use the maximum principle of elliptic PDEs theory.  First, he proved that for each point of the axis of the wedge, there exists a sphere centered at this point such  that the surface is invariant by the inversion about this sphere and next he showed  that a sphere is the only surface with this property \cite[Th. 3.2]{mc1,mc2}. Such as McCuan indicated in \cite{mc1}, the spherical reflection method can be of interest in other problems in global differential geometry of surfaces. Here we develop his technique on surfaces of constant mean curvature supported in a cone of Euclidean space.

We describe briefly the main ingredients of the spherical reflection method and we refer to the reader to \cite{mc1,mc2}.  Let $\s_{r}\subset\r^3$ be the sphere of radius $r$ centered at $O$. The spherical reflection about $\s_{r}$ is the inversion mapping defined by
$$\phi_r: \r^3\setminus\{O\}\rightarrow \r^3\setminus\{O\},\ \ \phi_r(p)= \frac{r^2}{|p|^2}p.$$
Here the center of inversion is taken as the origin $O$ and $r$ is  the radius of the inversion. If the radius $r$ is known in the context, we only write $\phi$ instead $\phi_r$. We introduce the next notation. If $p\in\r^3$, we denote  $\hat{p}=\phi(p)$. If $A\subset\r^3$, denote  $A(r)^{-}=\{p\in A: |p|> r\}$ and $A(r)^{+}=\{p\in A: |p|< r\}$. Let $\hat{A}(r)=\phi(A(r)^{-})$ be the image of $A(r)^{-}$ by  the spherical reflection about $\s_{r}$. We write $A\geq B$ if for every $p\in A$, $q\in B$ such that $p$ and $q$ lie in the same ray from the origin, we have $|p|\geq |q|$.

Let $S$ be an oriented surface of $\r^3$, $O\not\in S$ and let $N$ be the Gauss map of $S$. By a spherical reflection $\phi$, $\phi(S)$ is a surface where  the Gauss map $\hat{N}$ is
 \begin{equation}\label{gauss}
\hat{N}(\hat{p})=N(p)-\frac{2\langle N(p),p\rangle}{|p|^2}p.
\end{equation}
An easy calculation yields
$$\hat{\lambda}_i(\hat{p})=\frac{\lambda_i(p)|p|^2+2\langle N(p),p\rangle}{r^2},\ \ i=1,2,$$
where $\hat{\lambda}_i$ and $\lambda_i$ are the principal curvature of the surface $\phi(S)$ and $S$ respectively. Hence the mean curvature $\hat{H}$ of $\phi(S)$ at $\hat{p}$ is
\begin{equation}\label{mean}
\hat{H}(\hat{p})=\frac{H|p|^2+2\langle N(p),p\rangle}{r^2},
\end{equation}

For the second part of the spherical reflection method, we need the following result and whose proof done below simplifies  the given one   in \cite{mc2}.

\begin{proposition}\label{pinversion}
Let $S$ be a surface of constant mean curvature. If { as a set}, $S$ is invariant by an inversion   about a sphere $\s^2$, then $S$ is either part of a sphere or part of a plane passing through the center of the sphere.
\end{proposition}

\begin{proof}
After a rigid motion, we assume that $S$ is invariant by the inversion about the sphere $\s_r$. Since
$\hat{H}(\hat{p})$ is given by \eqref{mean} and  { $\phi(S(r)^-)=S(r)^{+}$}, then
\begin{equation}\label{mean2}
 |p|^2H+2\langle N(p),p\rangle=Hr^2
 \end{equation}
on $S$. At each point $p\in S$, let $\{e_1,e_2\}$ be a basis of principal directions, that is, $(dN)_p(e_i)=-\lambda_i(p)e_i$, $i=1,2$. Then the derivative of \eqref{mean2} along $e_i$ gives
\begin{equation}\label{mean3}
2(H-\lambda_i(p))\langle p,e_i\rangle=0,\ \ i=1,2.
\end{equation}
Assume that at a point $p\in S$, $  N(p)$ is not parallel to $p$. Then this occurs in an open set $V\subset S$ around $p$.
We conclude by \eqref{mean3} that { for some $i\in\{1,2\}$,  $\lambda_i=H$ on $V$. Since $2H=\lambda_1+\lambda_2$, we derive that $\lambda_1=\lambda_2=H$ on $V$ and all points of $V$ are umbilics.}  Because the umbilics in a non-umbilical surface of constant mean curvature are isolated \cite{ho}, we conclude that $S$ is part of a sphere or part of a plane  passing through the origin.  If $N(p)$ is parallel  to $p$ on $S$, that is,  $N(p)=\pm p/|p|$ on $S$, then one easily concludes that $S$ is a sphere centered at the origin $O$, obtaining the result again.
\end{proof}

We also need the tangency principle  based on the maximum principle of mean curvature equation.

\begin{proposition}[Tangency principle]\label{pmax} Let $S_1$ and $S_2$ be two surfaces that are expressed locally as graphs of functions $u_1$ and $u_2$ over the closure of a planar domain $\Omega\subset\r^2$. Assume $p_0\in S_1\cap S_2$ is a common tangent point and denote $p_0=u_i(x_0,y_0)$, $(x_0,y_0)\in\overline{\Omega}$. Suppose $p_0$ satisfies  one  of the next conditions:
\begin{enumerate}
\item $p_0\in int(S_1)\cap int(S_2)$, or equivalently, $(x_0,y_0)\in\Omega$ (interior point).
\item $p_0\in\partial S_1\cap \partial S_2$ and $\partial S_1$ and $\partial S_2$ are tangent at $p_0$. This is equivalent to $(x_0,y_0)\in \partial\Omega$ and $\partial\Omega$ is smooth at $(x_0,y_0)$ (boundary point).
    \item $(x_0,y_0)\in\partial\Omega$ is a point where $\Omega$ is a quadrant of $\r^2$ and $(x_0,y_0)$ is the corner of this quadrant (corner point).
\end{enumerate}
We orient $S_i$ by the   unit  normal vector fields such that $N_1(p_0)=N_2(p_0)$ and consider the reference system of $\r^3$ where the vector $N_i(p_0)$ indicates the positive vertical line. Denote by $H_i$ the mean curvature of $S_i$ according to the orientation given by $N_i$, $i=1,2$. If in a neighborhood of $p_0$, $S_1$ lies above $S_2$ and $H_1\leq H_2$, then $S_1$ coincides with $S_2$ in an open set around $p_0$.
\end{proposition}
 The first two cases are a direct consequence of the Hopf maximum principle \cite{eho}, usually called, the interior point version and the boundary point version of the tangency principle. See also \cite{sch}. The third case was shown by Serrin in \cite{se2} and is called the Serrin corner version of the tangency principle.

\section{Proof of Theorem \ref{tmain}}

Denote $\mathcal{S}=\partial\Omega$ and we assume that the vertex $O$ is the origin of coordinates. Recall that $\mathcal{S}=T\cup S$, where $T$ is the part of $\partial\Omega$ in the cone $C_\Gamma$.  Consider the family of spheres $\{\s_r: r>0\}$ centered at  $O$. Because $\overline{\Omega}$ is compact, consider $r$ so large such that $S(r)^{-}=\emptyset$.  We decrease the radius of the sphere $\s_r$ until the first value  $r_0>0$ such that $\s_{r_0}$ touches $\overline{\Omega}$. Because $S$ is embedded, for a small  $\epsilon>0$, we have
 $$\hat{\mathcal{S}}(r)\geq  \mathcal{S}(r)^{+}\ \ \mbox{and}\ \ \hat{\mathcal{S}}(r)\subset \Omega,$$
for all $r\in (r_0-\epsilon,r_0)$. For these values of $r$, the support function $f(p)=\langle N(p), p\rangle$ is negative for any $p\in S(r)^{-}$ because $N$ points to $\Omega$. Next, we decrease the radius of $\s_r$ and reflecting $\mathcal{S}(r)^{-}$ about $\s_r$. Let
$$r_1=\inf\{r:  \hat{\mathcal{S}}(r)\geq  \mathcal{S}(r)^{+}, \hat{\mathcal{S}}(r)\subset \Omega\}.$$
We have two possibilities:

{\it Case 1.} $r_1=0$. In such a case, for all $r\in (0,r_0)$, $\hat{\mathcal{S}}(r)\subset\Omega$. This shows that $S$ is a radial graph on $D$, proving the first part of Theorem.

{\it Case 2.} $r_1>0$.   We are going to prove that this case  is not possible. For this value $r_1$, there exists a touching point $p_0$  between $\mathcal{S}(r_1)^+$  and $\hat{\mathcal{S}}(r_1)$. Denote $q_0\in S(r_1)^{-}$ such that $p_0=\hat{q}_0=\phi_{r_1}({q_0})$. The possibilities  of touching point $p_0$ are the following (see Theorem 4 in \cite{mc1}):
\begin{enumerate}
\item[[T1]] The point $p_0$ belongs to $int(S)$ with $p_0\in int(S(r_1)^+)\cap int(\hat{S}(r_1))$. Then $|p_0|<r_1$ and  $f<0$ on $S(r)^{-}$ for $r\in [r_1,r_0]$. In particular, $p_0$ is common tangent point of $S(r_1)^+$ and $\hat{S}(r_1)$.
\item[[T2]] The point    $p_0$ belongs to $int(S)$ with $p_0\in \partial S(r_1)^+\cap \partial\hat{S}(r_1)$. In such case $|p_0|=r_1$, $f(p_0)=0$. Moreover,  $f<0$ on $S(r)^{-}$ for $r\in (r_1,r_0]$. Again  $p_0$ is common tangent point of $S(r_1)^+$ and $\hat{S}(r_1)$ and $\partial S(r_1)^+$ and $\partial\hat{S}(r_1)$ are tangent at $p_0$.
\item[[T3]] The point $p_0$ belongs to $\partial S$ with $|p_0|<r_1$. In particular, $p_0\in \partial\hat{S}(r_1)\cap\partial S(r_1)^{+}$.  We claim that the tangent planes of $\hat{S}(r_1)$ and $S(r_1)^{+}$  coincide. Denote by $N_{C}$ the Gauss map of $C_\Gamma$ pointing to $\Omega$. Let $\alpha=\alpha(s)$ be an arc-length parametrization of $\partial S$. At the boundary points $\partial S$, we take the orthonormal frame
    $$\left\{N_{C},  \frac{\alpha}{|\alpha|}, \textbf{b}:=N_{C}\times \frac{\alpha}{|\alpha|}\right\},$$
    where $\times$ denotes the cross product of $\r^3$. Then the expression of $N$ is
    $$N=(\cos\gamma) N_{C}+\lambda(s) \frac{\alpha}{|\alpha|}+\mu(s) \textbf{b},$$
    where $\gamma$ is the contact angle and $\lambda$ and $\mu$ satisfy $\lambda^2+\mu^2=(\sin\gamma)^2$.
By (\ref{gauss}),
\begin{equation}\label{b1}
\hat{N}(s)=(\cos\gamma) N_{C}(s)-\lambda(s) \frac{\alpha(s)}{|\alpha(s)|}+\mu(s) \textbf{b}(s).
\end{equation}
Let $p_0=\hat{\alpha}(s_0)=\phi(\alpha(s_0))$.
Since $q_0=\alpha(s_0)$ and $\hat{\alpha}(s_0)$ lie on the same ray starting from $O$,
\begin{equation}\label{n1}
N(p_0)=N(\hat{\alpha}(s_0))=(\cos\gamma) N_{C}(s_0)+\lambda_1 \frac{\alpha}{|\alpha|}(s_0)+\mu_1 \textbf{b}(s_0),
\end{equation}
with $\lambda_1,\mu_1\in\r$. Recall
$$\phi(p)= \frac{r_{1}^2}{|p|^2}p.$$
 Then
\begin{equation}\label{n2}
d\phi(\alpha'(s))=\frac{r_{1}^2}{|\alpha(s)|^2}\alpha'(s)-2{r_1^2}\frac{\langle \alpha(s), \alpha'(s)\rangle}{|\alpha(s)|^4}\alpha(s),
\end{equation}
 and from \eqref{b1},
\begin{equation}\label{n3}
0=\langle\hat{N}(s), d\phi(\alpha'(s))\rangle=\lambda(s)\frac{r_{1}^2}{|\alpha(s)|^3}\langle \alpha(s), \alpha'(s)\rangle +\mu(s)\frac{r_{1}^2}{|\alpha(s)|^2}\langle \textbf{b}(s), \alpha'(s)\rangle.
\end{equation}
Since $|p_0|<r_1$, the contact between $\partial\hat{S}(r_1)$ and $\partial S(r_1)^{+}$ at $p_0$ is tangential. Then $d\phi(\alpha'(s_0))$ is parallel to $\hat{\alpha}'(s_0)$ and \eqref{n2} implies
\begin{eqnarray}\label{n4}
0&=&\langle N(\hat{\alpha}(s_0)), d\phi(\alpha'(s_0))\rangle\nonumber\\
&=&-\lambda_1\frac{r_{1}^2}{|\alpha(s_0)|^3}\langle \alpha(s_0), \alpha'(s_0)\rangle +\mu_1\frac{r_{1}^2}{|\alpha(s_0)|^2}\langle\textbf{b}(s_0), \alpha'(s_0)\rangle.
\end{eqnarray}
As $\langle \alpha(s_0), \alpha'(s_0)\rangle$ and $\langle\textbf{b}(s_0), \alpha'(s_0)\rangle$ can not 	simultaneously vanish because it would imply $\alpha'(s_0)=0$, we conclude from \eqref{n3} and \eqref{n4} that  $\mu_1\lambda(s_0)+\lambda_1\mu(s_0)=0$. As $\lambda(s_0)^2+\mu(s_0)^2= (\sin\gamma)^2=\lambda_1^2+\mu_1^2$, then  $\lambda(s_0)=-\lambda_1$ and $\mu(s_0)=\mu_1$. Then  \eqref{b1} and \eqref{n1} imply
$N(\hat{\alpha}(s_0))=\hat{N}(s_0)$.
Thus the tangent planes of $\hat{S}(r_1)$ and $S(r_1)^{+}$ coincide at $p_0=\hat{\alpha}(s_0)$. This completes the proof of the claim.
\item[[T4]] The point  $p_0$ belongs to $\partial S$ with $| p_0|=r_1$. Then $p_0=q_0$. The same reasoning as above proves both   $\hat{S}(r_1)$ and $S(r_1)^{+}$ as well as   $\partial\hat{S}(r_1)$ and $\partial S(r_1)^{+}$ are tangent at $p_0$.
\end{enumerate}
In order to apply the tangency principle (Proposition \ref{pmax}), we compare the mean curvatures of $\hat{S}(r_1)$ and  $S(r_1)^+$ in a neighborhood of the contact point $p_0$. Denote by $\hat{H}$ the mean curvature of $\hat{S}(r_1)$ with respect to the unit normal vector $\hat{N}$. Using \eqref{gauss}, we have that both $N$ and $\hat{N}$ point towards $\Omega$ because $f<0$ for any $r\in (r_1,r_0)$.  On the other hand,  remark that for all $r\in [r_1,r_0]$ and  $q\in S(r)^{-}$, we have
\begin{equation}\label{hh}
\hat{H}(\hat{q})=\frac{|q|^2H+\langle N(q),q\rangle}{r_1^2}\leq \frac{|q|^2}{r_1^2}H\leq H,
\end{equation}
because $H\leq 0$ on $S$ and $|q|\geq r_1$. In the McCuan's notion, the inequality \eqref{hh} is expressed by saying that $S$ satisfies the boundedness property \cite{mc1}.

We analyse the different cases of touching points that have appeared in the previous discussion.
\begin{enumerate}
\item[[T1]]  Because $\hat{S}(r_1)\geq S(r_1)^{+}$ around $p_0$ and $N(p_0)=\hat{N}(p_0)$ point to $\Omega$, the surface $\hat{S}(r_1)$ lies above $S(r_1)^{+}$ in a neighborhood of the point $p_0$. In this neighborhood,  $\hat{H}\leq H$ and  the tangency principle implies that $\hat{S}(r_1)=S(r_1)^+$ in an open set around $p_0$. Repeating the reasoning, it is proved that the set of points where $\hat{S}(r_1)$ coincides with $S(r_1)^{+}$ is an open set in $S$, that is, $S$ is invariant by the inversion about $\s_{r_1}$.
\item[[T2]] This case is similar to the above one with the difference that $p_0$ is a boundary point of $\hat{S}(r_1)$ and $S(r_1)^+$. The tangency principle in its boundary point  version concludes that $S$ is   invariant by the inversion about $\s_{r_1}$.
\item [[T3]] This case is the same the above one because $|p|<r_1$ and thus, it is a common boundary point of $\hat{S}(r_1)$ and $S(r_1)^+$.
\item[[T4]] Now the surfaces $\hat{S}(r_1)$ and  $S(r_1)^+$ are tangent at a common boundary point  $p_0$ where  both surfaces may be expressed locally as a graph over a quadrant   in the common tangent plane at $p_0$. We apply  the Serrin corner point  version of the tangency principle.
\end{enumerate}
In all cases, an argument using connectedness proves finally that $S$ is invariant by the inversion about ${\s}_{r_1}$.  Clearly, any plane passing through the origin can not be a capillary surface in the cone $C_D$ and, finally,  Proposition \ref{pinversion} gives that $S$ is part of a sphere. Since $S$ is invariant by the inversion about $\s_{r_1}$, $S$ is either part of $\s_{r_1}$ or part of sphere which is not containing the origin inside of $S$. With the unit normal vector pointing toward bounded domain, the mean curvature is strictly positive. This contradicts the condition on the mean curvature.
An analysis of the different configurations of  a capillary sphere in a cone $C_D$ with non-positive mean curvature with respect to the unit normal vector field pointing $\Omega$ gives  that $C_\Gamma$ is a circular cone and $S$ is a concave sphere (Figure \ref{fig2}, case (a)).  But this sphere is not invariant by inversions about a sphere of type $\s_r$. This proves that the case 2) does not occur.

We prove the last statement of Theorem \ref{tmain}. Assume now that $C_\Gamma$ is a circular cone. By the previous argument, we know that $S$ is a radial graph on $D$.  We claim that $\partial S$ is a circle that lies a plane which is perpendicular to the axis $L$ of the circular cone $C_{\Gamma}$.

Suppose on the contrary. Denote by  $p_{M}$ and $p_m$ the points of $\partial S$  which lie at the maximum and minimum distance from $O$. Let $l_ M$ and $l_m$   the half-lines from $O$ passing through $p_M$ and $p_m$ respectively.   Since $\partial S$ is not a circle and it is a radial graph on $\Gamma$, $l_M\neq l_m$. Rotate $S$ about $L$ until $l_M$ coincides with $l_m$ and we call $S^*$ the new position of $S$. If $A\subset\r^3$ is a subset of $\r^3$, denote by  $\rho A$   the image of $A$ by the homothety of ratio $\rho>0$. Let $\rho<1$ be small enough such that $\rho S^*$ is disjoint to $S$.   By  homothetically expanding $\rho S^*$,  that is, increasing the value of $\rho$, we arrive until the first value $\rho=\rho_1$ such that $\rho_1 S^{*}$ touches $S$. Because $|p_m|<|p_M|$, the value $\rho_1$ satisfies $\rho_1<1$. We analyze the different types of contact points between $\rho_1 S^{*}$ and $S$.

If the contact point $p_0$  occurs at some interior point between $\rho_1 S^*$ and $S$, then the unit normal vectors of $\rho_1 S^{*}$ and $S$ agree at $p_0$ because both vectors point in the opposite direction of $p_0$. Around $p_0$, consider the reference system whose vertical direction is given by the unit normal vector at $p_0$. With respect to this reference system, the surface $\rho_1 S^{*}$ lies above $S$ in a neighborhood of $p_0$. However, the mean curvature of $\rho_1S^*$ is $H/\rho_1$ and because $H\leq 0$, we have
$H/\rho_1<H$ if $H<0$ or $H/\rho_1=H$ if $H=0$. This is a contradiction with the  tangency principle.

If the touching point $p_0$ is a boundary point, then $\rho_1 S^*$ and $S$ are tangent at $p_0$ because the constancy of the contact angle $\gamma$ (homotheties and rotations about $L$ do not change the value of $\gamma$). We apply again the  tangency principle, obtaining a contradiction again. As a conclusion, $|p_M|=|p_m|$ and  every point of $\partial S$ has the same distance from $O$. Because $\Gamma$ is a circle, then $\partial S$ is a circle too. This completes the proof of the claim.

Once proved that $\partial S$ is a circle, the proof of Theorem \ref{tmain} ends as follows. If $H=0$, $S$ is a planar disc by the convex hull property { that satisfy minimal surfaces. Suppose now $H\neq 0$}. Let $\Pi$ be the plane containing the circle $\partial S$. Since $S$ is a radial graph, $S\cap (\Pi\setminus\overline{W})=\emptyset$ where $W\subset\Pi$ is the   disc bounded by $\partial S$. A result due to Koiso \cite{ko} shows that $S$ lies in one side of $\Pi$ and finally, the Alexandrov reflection method by { using planes orthogonal to $\Pi$} proves that $S$ is a spherical cap. This completes the proof of Theorem \ref{tmain}.  $\Box$

As a consequence of this theorem, we   give some partial answers to the questions posed in the Introduction.  Using that   Theorem \ref{tmain} asserts that the boundary curve $\partial S$ is homologous to $\Gamma$ in $C_\Gamma\setminus\{O\}$, we obtain:

\begin{corollary} There are no   capillary surfaces on a cone $C_D$ with non-positive mean curvature whose boundary is nullhomologous in $C_\Gamma\setminus\{O\}$.
\end{corollary}

\begin{corollary} There are no capillary bridges on a cone with non-positive mean curvature.
\end{corollary}

We point out that for some contact angles, there exist capillary spherical domains in non-circular cones. For instance, this occurs when the contact angle $\gamma=\pi/2$. For this, take a domain $D$ in a hemisphere of $\s^2$ and consider $D_\rho=\rho D$ the image of $D$ by the homothety of ratio $\rho>0$. Then $D_\rho$ is a capillary surface in the cone $C_D$ intersecting orthogonally $C_{\Gamma}$. In this example, the mean curvature of $D_\rho$ is positive with respect to the unit normal vector pointing to $\Omega$.

The next result has its origin in Propositions 2.1 and 2.2 of \cite{vo}, but we do not assume that $\overline{\Omega}$ wets the origin.

\begin{corollary}\label{prh} Let $C_D$ be a circular cone of opening angle $2\varphi$.   Let $S$ be a  connected   capillary surface $S$ in $C_D$  whose unit normal vector field $N$ points to $\Omega$. Let $\gamma$ be the contact angle.
\begin{enumerate}
\item If $\gamma<(\pi+\varphi)/2$, then the mean curvature $H$   is positive.
\item If $\gamma>(\pi+\varphi)/2$, then $S$ is a concave spherical cap.
\item If $\gamma=(\pi+\varphi)/2$, then $S$ is a planar domain.
\end{enumerate}
\end{corollary}

\begin{proof}
For each $\gamma$, there exists an  umbilical surface $U_{\gamma}$ in $C_D$ spanning $\Gamma$ and meeting $C_\Gamma$ with angle $\gamma$. Consider on   $U_\gamma$ the unit normal vector pointing to  the  domain $\Omega_\gamma$ bounded by $U_\gamma$ and $C_\Gamma$. Because $U_\gamma\subset C_D$, then $0\leq\gamma\leq\pi$. In fact, the surface $U_\gamma$ is a planar disc if $\gamma=(\pi+\varphi)/2$ and
a spherical cap if $\gamma\not=(\pi+\varphi)/2$. See Figure \ref{fig2}. Moreover,
 \begin{enumerate}
\item[(i)] If $\gamma<(\pi+\varphi)/2$, $U_\gamma$ lies over
the plane $\Pi$ containing $\Gamma$ and the mean curvature is positive. If $\gamma<\pi/2-\varphi$, then $U_\gamma$ is a large spherical cap and if $\gamma>\pi/2-\varphi$, $U_\gamma$ is a small spherical cap. In the case $\gamma=\pi/2-\varphi$, $U_\gamma$ is a hemisphere.
\item[(ii)] If $\gamma>(\pi+\varphi)/2$, $U_\gamma$ lies below
the plane $\Pi$ containing $\Gamma$ and the mean curvature is negative.
\end{enumerate}
{Assume that $H\leq 0$. Then Theorem \ref{tmain} implies that $S$ is a planar disc or a spherical cap. By the above description of the configurations of all $U_\gamma$ depending on the contact angle $\gamma$, the only possibilities are the cases (2) and (3). As a conclusion, if $\gamma<(\pi+\varphi)/2$, then $H$ must be positive.}
\end{proof}

Finally, we obtain a result of symmetry using the Alexandrov reflection method by planes.

\begin{proposition} Consider a circular cone $C_\Gamma$ whose axis $L$ is the  vertical $z$-axis of $\r^3$ and let $S$ be a capillary surface in $C_D$. Assume that the boundary $\partial S$ lies in a half-cone, that is, $\partial S\subset \{p\in C_\Gamma: \langle p,a\rangle>0\}$ for some unit horizontal vector $a\in\r^3$. Then $S$ is symmetric with respect to a vertical plane.
\end{proposition}

\begin{proof} Without loss of generality, we assume that $a=(1,0,0)$. As usual, let $\Omega\subset C_D$ be the bounded domain by $S$ and $C_\Gamma$ and let choose on  $S$  the Gauss map $N$ that points towards $\Omega$. Consider  the planes $P(t)$ of equation $x=t$ for $t\in (-\infty,0]$. We view the $x$-axis as an oriented line where the righ direction points towards $x\rightarrow+\infty$.

We start with the Alexandrov method with reflections about the planes $P(t)$ from $t=-\infty$. We move $P(t)$ to the right by $t\nearrow 0$   until the first contact with $S$. {A possible case is that the contact occurs at $t=t_0<0$. We} move $P(t)$ slightly to the right  and we reflect the left side of $S$ about $P(t)$. For values $t\in (t_0,t_0+\epsilon)$, the reflected surface lies in $\Omega$. We move $P(t)$ on the right and reflecting the part of $S$ on the left with respect to $P(t)$. If there is a contact point at $t=t_1$ with $t_1\leq 0$ and $t_0<t_1$,   the touching point must be interior by the geometry of the circular cone $C_\Gamma$ and because the boundary $\partial S$ lies in the half-space $\r^3_{a}=\{(x,y,z)\in\r^3: x>0\}$. Then   the tangency principle implies that $S$ has a symmetry plane $P(\bar{t})$ for some $\bar{t}\leq 0$, which it is impossible since $\partial S$ lies in the half-space $\r^3_{a}$. Moreover, by the geometry of $C_\Gamma$ there is not contact point with some boundary point. This means that we arrive until $t=0$ and the reflected surface lies in $\Omega$.  We call this last possibility case (i).

Case (ii).  The other possibility is  that when we move $P(t)$ from $t=-\infty$, there is no contact point between $P(t)$ and $S$ for all $t\leq 0$.

 In both cases (i) and (ii), we arrive with the planes $P(t)$ until the position  $t=0$ and we stop here.

Now we change the planes $P(t)$ in the reflection method by the one parameter family of planes   $Q(t)$ that contain the axis $L$ and parametrized by the angle $t\in [0,\pi]$ that makes $Q(t)$ with the vector $(1,0,0)$. Exactly,
 $$Q(t)=\{(x,y,z)\in\r^3: \cos(t) x+\sin(t)y=0\}.$$
Introduce the next notation:
$$S_Q(t)^{-}=\{(x,y,z)\in S:  \cos(t) x+\sin(t)y<0\},$$
$$S_Q(t)^{+}=\{(x,y,z)\in S:  \cos(t) x+\sin(t)y>0\}$$
and $\tilde{S}_Q(t)$ the reflection of $S_Q(t)^{-}$ about $Q(t)$.  For $t=0,\pi$, $Q(0)=P(0)=Q(\pi)$. By the process followed until here, we have two possibilities according to the above discussion of cases (i) and (ii):
 \begin{enumerate}
 \item\emph{Case (i)}. We know that $P(t)$ touches the first time $S$ at $t=t_0<0$ and by the argument of moving to the right and reflecting, the reflected surface about the plane $P(0)$ lies in the domain $\Omega$. At this position, indeed, when we stop at $P(0)$, we change by the reflection process about $Q(t)$, that is,  we begin reflecting $S_Q(t)^{-}$ about $Q(t)$. Remark that the reflection about $Q(t)$ leaves invariant $C_\Gamma$ because $Q(t)$ contains the axis $L$ of $C_\Gamma$. For values $t$ in some interval  $[0,\epsilon)$ around $t=0$, the reflected surface $\tilde{S}_Q(t)$ lies in the domain $\Omega$. Letting $t\nearrow  \pi$, we arrive the first time $t=t_1<\pi$ such that $\tilde{S}_Q(t_1)$ has a first contact point $p_0$ with $S_Q(t_1)^{+}$. This occurs because $\partial S\subset \r^3_{a}$ and the planes  $Q(t)$ sweep   the half-cone $C_\Gamma\cap \r^3_{a}$.  If $p_0$ is an interior point, then the  interior tangency principle implies that $Q(t_1)$ is a plane of symmetry of $S$, proving the result. If $p_0\in\partial S$, then the tangent planes of $S$ and the reflected surface agree at $p_0$ by the constancy of the contact angle and because reflections do not change this angle. The tangency principle   of the boundary point version or corner point version concludes that $Q(t_1)$ is a plane of symmetry of $S$ again.
 \item \emph{Case (ii)}. By moving $P(t)$ from $t=-\infty$,  we arrive  until $t=0$ and we have no contact with $S$. Then we do a similar argument of reflection with the planes $Q(t)$ as above. The difference now is that firstly we arrive until a first position that the planes $Q(t)$ touches $S$ at some time $t=t_o\in (0,\pi)$. Next, we follow increasing $t\nearrow \pi$ and reflecting as in the above case, until that we find a time $t=t_1\in (0,\pi)$, $t_1>t_0$, such that $\tilde{S}_Q(t_1)$  touches the first time with   $S_Q(t_1)^{+}$. The same argument   for the case (i) proves that $Q(t_1)$ is a plane of symmetry of $S$.
     \end{enumerate}
\end{proof}

\textbf{Acknowledgement.} The first author  is partially supported by MEC-FEDER
 grant no. MTM2011-22547 and
Junta de Andaluc\'{\i}a grant no. P09-FQM-5088.  The second author was supported in part by the National Research Foundation of Korea (2012-0007728) and the Korea Institute for Advanced Study (KIAS) as an Associate Member.


\begin{thebibliography}{123}

 \bibitem{al}   A.D. Alexandrov,
  A characteristic property of spheres, Ann. Mat. Pura Appl.  58 (1962) 303--315.

\bibitem{cp}   J. Choe, S. Park,   Capillary surfaces in a convex cone,  Math. Z. 267 (2011) 875--886.

\bibitem{fi}   R. Finn,
Equilibrium Capillary Surfaces, Springer-Verlag, Berlin,  1986.

\bibitem{eho} E. Hopf,
Elementare Bemerkumgen \"{U}ber Die L\"{o}sungen Partieller Differential Gleichungen. S. B. Preuss. Akad.  Phys. Math. Kl. 147--152 (1927)

\bibitem{ho}   H. Hopf,
Differential Geometry in the Large, Lecture Notes in Mathematics,  1000, Springer-Verlag, Berlin, 1983.

 \bibitem{ko}   M. Koiso,    Symmetry of hypersurfaces of constant mean curvature with symmetric boundary. Math Z. 191 (1986) 567--574.

\bibitem{kn} G.P. Kubalski, M.   Napiorkowski,
A liquid drop in a cone-line tension effects, J. Phys.: Condens. Matter 12 (2000) 9221--9229.



 \bibitem{lo} R. L\'opez,  A note on radial graphs with constant mean curvature,  Manuscripta Math. 110  (2003) 45--54.

\bibitem{mc1}  J. McCuan,   Symmetry via spherical reflection and spanning drops in a wedge,  Pacific J. Math. 180 (1997) 291--323.

\bibitem{mc2}  J. McCuan,
  Symmetry via spherical reflection,  J. Geom. Anal. 10 (2000) 545--564.


\bibitem{ra} T. Rad\'o, Contributions to the theory of minimal surfaces, Acta Litt. Sci. Univ. Szeged 2 (1932) 1--20.

\bibitem{rv}  A. Ros, E.  Vergasta,   Stability for hypersurfaces of constant mean curvature with free boundary,
Geom. Dedicata, 56 (1995) 19-33.

\bibitem{sch} R. Schoen,
Uniqueness, symmetry, and embeddedness of minimal surfaces,
J. Differential Geom.  18 (1983) 791--809.

\bibitem{sc} D. Schwab, Interior regularity of conical capillary surfaces, Pacific J. Math. 229 (2007) 499--510.


\bibitem{se}  J.  Serrin,    The problem of Dirichlet for quasilinear elliptic equations with many independent variables, Phylos. Trans. Roy. Soc. London A   264 (1969) 413--496.



\bibitem{se2}  J.  Serrin,  A symmetry problem in potential theory, Arch. Rat. Mech. Anal.   43 (1971) 304--318.


\bibitem{ta} E. Tausch,  The $n$-dimensional least area problem for boundaries on a convex cone, Arch. Ration. Mech. Anal. 75 (1981) 407--416.

\bibitem{vo}  T. Vogel,   Uniqueness of capillary surfaces in wedges and cones, Geometric Analysis and Nonlinear Partial Differential Equations, Lecture Notes in Pure and Appl. Math. 144, 129--138. Dekker, New York, 1993.
\end{thebibliography}
\end{document}